\documentclass[10pt]{amsart}
\usepackage{graphicx}
\usepackage{amssymb}
\usepackage{xy}
\xyoption{all}
\usepackage{soul}
\usepackage{psfrag}

\usepackage[
paper=b5paper,
text={140mm,203mm},centering
]{geometry}

\iftrue
\makeatletter
\def\@settitle{%
  \vspace*{-20pt}
  \begin{flushleft}%
    \baselineskip14\p@\relax
    \normalfont\bfseries\LARGE
    \@title
  \end{flushleft}%
}
\def\@setauthors{%
  \begingroup
  \def\thanks{\protect\thanks@warning}%
  \trivlist
  \large \@topsep30\p@\relax
  \advance\@topsep by -\baselineskip
  \item\relax
  \author@andify\authors
  \def\\{\protect\linebreak}%
  \authors
  \ifx\@empty\contribs
  \else
    ,\penalty-3 \space \@setcontribs
    \@closetoccontribs
  \fi
  \vskip1ex
  \normalfont
  \@setaddresses
  \endtrivlist
  \endgroup
}
\def\@setaddresses{\par
  \nobreak \begingroup
  \small
  \def\author##1{\nobreak\addvspace\smallskipamount}%
  \def\\{\unskip, \ignorespaces}%
  \interlinepenalty\@M
  \def\address##1##2{\begingroup
    \par\addvspace\bigskipamount\noindent
    \@ifnotempty{##1}{(\ignorespaces##1\unskip) }%
    {\ignorespaces##2}\par\endgroup}%
  \def\curraddr##1##2{\begingroup
    \@ifnotempty{##2}{\nobreak\noindent\curraddrname
      \@ifnotempty{##1}{, \ignorespaces##1\unskip}\/:\space
      ##2\par}\endgroup}%
  \def\email##1##2{\begingroup
    \@ifnotempty{##2}{\nobreak\noindent E-mail address%
      \@ifnotempty{##1}{, \ignorespaces##1\unskip}\/:\space
      \ttfamily##2\par}\endgroup}%
  \def\urladdr##1##2{\begingroup
    \def~{\char`\~}%
    \@ifnotempty{##2}{\nobreak\noindent\urladdrname
      \@ifnotempty{##1}{, \ignorespaces##1\unskip}\/:\space
      \ttfamily##2\par}\endgroup}%
  \addresses
  \endgroup
  \global\let\addresses=\@empty
}
\def\@setabstracta{%
    \ifvoid\abstractbox
  \else
    \skip@25\p@ \advance\skip@-\lastskip
    \advance\skip@-\baselineskip \vskip\skip@
    \box\abstractbox
    \prevdepth\z@ 
    \vskip-10pt
  \fi
}
\renewenvironment{abstract}{%
  \ifx\maketitle\relax
    \ClassWarning{\@classname}{Abstract should precede
      \protect\maketitle\space in AMS document classes; reported}%
  \fi
  \global\setbox\abstractbox=\vtop \bgroup
    \normalfont\small
    \list{}{\labelwidth\z@
      \leftmargin0pc \rightmargin\leftmargin
      \listparindent\normalparindent \itemindent\z@
      \parsep\z@ \@plus\p@
      
    }%
    \item[\hskip\labelsep\bfseries\abstractname.]%
}{%
  \endlist\egroup
  \ifx\@setabstract\relax \@setabstracta \fi
}
\def\section{\@startsection{section}{1}%
  \z@{-1.2\linespacing\@plus-.5\linespacing}{.8\linespacing}%
  {\normalfont\bfseries\Large}}
\def\subsection{\@startsection{subsection}{2}%
  \z@{-.8\linespacing\@plus-.3\linespacing}{.3\linespacing\@plus.2\linespacing}%
  {\normalfont\bfseries}}
\def\subsubsection{\@startsection{subsection}{3}%
  \z@{.7\linespacing\@plus.2\linespacing}{-1.5ex}%
  {\normalfont\itshape}}
\def\@secnumfont{\bfseries}
\makeatother
\fi 

\newcommand{\cc}{\breve{C}}
\newcommand{\lc}{\breve{L}}
\newcommand{\N}{\mathbb N}
\newcommand{\Z}{\mathbb Z}

\newcommand{\OO}{\mathcal O}

\def\spinc{\textrm{Spin}^c}

\def\Wh{\operatorname{Wh}}

\newtheorem{theorem}{Theorem}[section]
\newtheorem{thm}{Theorem}

\newtheorem{lemma}[theorem]{Lemma}

\newtheorem{proposition}[theorem]{Proposition}
\newtheorem{corollary}[theorem]{Corollary}%
\theoremstyle{definition}

\title
{Smooth concordance of links topologically concordant
to the Hopf link}
\author[Jae Choon Cha]{Jae Choon Cha}
\address{Department of Mathematics and PMI\\
  POSTECH \\
  Pohang 790--784\\
  Republic of Korea}
\email{\rm{jccha@postech.ac.kr}}
\author[Taehee Kim]{Taehee Kim}
\address{Department of Mathematics\\
  Konkuk University \\
  Seoul 143--701 \\
  Republic of Korea}
\email{\rm{tkim@konkuk.ac.kr}}
\author[Daniel Ruberman]{Daniel Ruberman}
\address{Department of Mathematics, MS 050\\
Brandeis University\\
Waltham, MA 02454\\
USA}
\email{\rm{ruberman@brandeis.edu}}
\author[Sa\v so Strle]{Sa\v so Strle}
\address{Faculty of Mathematics and Physics\\
University of Ljubljana\\
Jadranska 21\\
1000 Ljubljana, Slovenia }
\email{\rm{saso.strle@fmf.uni-lj.si}}
\renewcommand{\phi}{\varphi}

\def\subjclassname{\textup{2010} Mathematics Subject Classification}
\expandafter\let\csname subjclassname@1991\endcsname=\subjclassname
\expandafter\let\csname subjclassname@2000\endcsname=\subjclassname

\begin{document}

\begin{abstract}
It was shown by Jim Davis that a $2$-component link with Alexander polynomial one is topologically concordant to the Hopf link. In this paper, we show that there is a $2$-component link with Alexander polynomial one that has unknotted components and is not smoothly concordant to the Hopf link, answering a question of Jim Davis. We construct infinitely many concordance classes of such links, and show that they have the stronger property of not being smoothly concordant to the Hopf link with knots tied in the components.
\end{abstract}
\maketitle

\section{Introduction}
The study of odd-dimensional link concordance has complications that go beyond the study of knot concordance of the individual components.  In this paper, we discuss some additional differences that arise in the classical dimension when one also considers the distinction between smooth and topological concordance.  We consider the question of whether a $2$-component link is smoothly concordant to the Hopf link. A well-known theorem of M.~Freedman~\cite{freedman:non-simply-connected,freedman-quinn} states that a knot whose Alexander polynomial is one is topologically concordant to the trivial knot.  
Completing a program initiated by J. A. Hillman~\cite{hillman:link-invariants}, J.~Davis~\cite{davis:hopf} showed that a $2$-component link with (multivariable) Alexander polynomial one is topologically concordant to the Hopf link.

It follows directly from the existence of smoothly non-slice knots with Alexander polynomial one that Davis' theorem cannot hold in the smooth category. A ``more refined'' question in this setting was posed by Davis: is there a $2$-component link with Alexander polynomial one which is not smoothly concordant to the Hopf link, but each of whose components is smoothly concordant to the unknot?
In Section~\ref{section:proof-main}, we provide such a link; in fact
we prove a stronger result, whose statement benefits from a bit of
terminology.

Given a link $L=(L_1,\ldots,L_n)$, and a split link $L'=(L_1',\ldots,L_n')$, one can form the connected sum $L \# L'=(L_1\# L_1',\ldots,L_n\# L_n')$.   Since $L'$ is split, the connected sum is well-defined.   We will refer to the resulting link as a {\em locally knotted $L$}. Let $H$ denote the Hopf link and $L'$ a $2$-component link. If both components of $L'$ are knots of trivial Alexander polynomial, then
the locally knotted Hopf link $H \#L'$ will have (multivariable) Alexander polynomial equal to one.  Note that if one component of $L'$ is not smoothly slice, then $H \#L'$ is certainly not smoothly concordant to $H$; this is the point of Davis' question.
\begin{thm}\label{T:alex1} There is a $2$-component link with Alexander polynomial one which
has unknotted components and is not smoothly concordant to any locally knotted Hopf link.
\end{thm}

In Section~\ref{section:proof-main}, we give two proofs for
Theorem~\ref{T:alex1} which use the Ozsv{\'a}th-Szab{\'o} $\tau$
invariant~\cite{Ozsvath-Szabo:2003-1}.
In Theorem~\ref{T:infinite} using the $d$-invariant (or correction
term~\cite{oz:boundary}) we will in fact show that there are
infinitely many smooth concordance classes of such links.

\medskip
\begin{trivlist}
\item[] \textbf{Acknowledgments:} This paper was inspired by conversations
over the years with Jim Davis about his result on concordance to the
Hopf link.  We thank Tim Cochran, Matt Hedden, Chuck Livingston, and
Adam Levine for helpful discussions.
The first author was supported
by the National Research Foundation of Korea(NRF) grant funded by
the Korea government(MEST) (No.\ 2010--0011629 and 2009--0094069).  
The second author was supported
by the National Research Foundation of Korea(NRF) grant funded by
the Korea government(MEST) (No.\ 2010--0015514 and 2010--0001194). The
third author was partially supported by NSF Grant 0804760.   The
fourth author was  supported in part by the Slovenian Research
Agency program No. P1--0292--0101.  Visits of the authors were
supported by a Slovenian-U.S.A. Research Project BI--US/09--12--004,
and by NSF Grant 0813619.
\end{trivlist}

\section{Covering links and blow-down}
\label{section:covering-links-and-blowdown}

All links will be assumed to be oriented. Generally speaking, concordance will refer to smooth concordance, with the adjective `topological' (always meaning topologically 
locally flat) added as appropriate.  Links are always ordered.  We will generally use the same letters for a link and its components, so that for example $L_1$ and $L_2$ would indicate the first and second components of a $2$-component link~$L$.

We offer two related proofs of Theorem~\ref{T:alex1}.  The first uses
the technique of covering links~\cite{Cha-Kim:2008-1} while the second
comes from observations on blow-down for links. 

\subsection*{Covering link calculus}

In this paper we will use the following construction, which is a
special case of covering link calculus formulated more generally
in~\cite{Cha-Kim:2008-1} for links with arbitrary number of components
in $\Z_p$-homology spheres (see also
\cite{Cochran-Orr:1993-1,Cha-Ko:2006-1,Cha-Livingston-Ruberman:2006-1}).

Let $p$ be a prime.  For a link $L=L_1\cup L_2$ in $S^3$, consider the
$p^a$-fold cyclic cover of $S^3$ branched along~$L_2$, say $Y$, and
then consider a component, say $J$, of the pre-image of~$L_1$.
Viewing $J$ as a knot in $Y$, we call $J$ a \emph{covering knot}
of~$L$.  Though the construction and the lemma below apply to more
general cases, in this paper we will always apply these to a 2-component
link $L$ in $S^3$ with $L_2$ unknotted and
$\operatorname{lk}(L_1,L_2)=1$, so that $Y$ is $S^3$ again and $J$ is
the whole pre-image of~$L_1$.

The following is a well-known fact, which holds in both topological
and smooth category.

\begin{lemma}[e.g., see \cite{Cochran-Orr:1993-1,Cha-Ko:2006-1,Cha-Kim:2008-1}] 
  \label{lemma:covering-knot}
  Suppose $L$ and $L'$ are concordant links.  Then their $p^a$-fold
  covering knots $J$ and $J'$ are rationally concordant in the sense
  of \cite{Cha-Ko:2000-1,Cha:2003-1}, i.e., concordant in a rational
  homology $S^3\times I$.
\end{lemma}

\begin{corollary}
  \label{corollary:covering-knot-hopf-link}
  Suppose $L$ is concordant to the Hopf link.  Then the $p^a$-fold
  covering knot of $L$ is rationally slice, i.e. bounds a 2-disk in a
  rational homology 4-ball.
\end{corollary}

\begin{proof}
  A covering knot of the Hopf link is obviously unknotted.  From
  Lemma~\ref{lemma:covering-knot}, the conclusion follows immediately.
\end{proof}

\subsection*{Blow-down for links}

 Recall that the result of $\pm 1$ surgery on an unknot in $S^3$ is again $S^3$.  If $L$ is a link with an unknotted component $L_n$, then doing $\pm 1$ surgery on $L_n$ produces a new link $\lc$ in $S^3$.  We say that $\lc$ is obtained by blowing down $L_n$ (with framing specified as necessary).
\begin{lemma}
Suppose $L$ and $L'$ are concordant, and that $L_n$ and $L_n'$ are both unknots.   Then $\lc$ and $\lc'$ are concordant in a homotopy $S^3 \times I$.
\end{lemma}\label{L:blow1}
The lemma holds in either the topological or smooth category; in the topological category we can of course replace the homotopy $S^3 \times I$ by the real one.
\begin{proof} Let $C = (C_1,C_2,\ldots)$ be the concordance in $S^3 \times I$.  Following Gordon's classic paper~\cite{gordon:contractible} we can do the $\pm 1$ surgery on the component $C_n$, to produce a simply-connected homology cobordism between $S^3$ and itself. Since the surgery took place in the complement of $\bigcup_{i\ne n} C_i$, those components give a concordance~$\cc$.
\end{proof}
We will say that the concordance  $\cc$ is obtained by blowing down $C_n$.  In the special case when $L'= H$, the Hopf link, note that $\lc'_1$ is the unknot.  Hence we obtain an immediate corollary.

\begin{corollary}\label{C:blowdown}  Let $L$ be a $2$-component link such that the component $L_2$ is unknotted.  If $L$ is concordant to the Hopf link,  then the knot $\lc_1$ obtained by blowing down $L_2$ is {\em homotopically slice}, that is, slice in a homotopy $4$-ball.
\end{corollary}

\subsection*{Local knotting}

For a link $L=(L_1,\ldots,L_n)$ let $S(L)$ denote the split link with the same components as $L$, thus $S(L)=L_1\coprod\cdots\coprod L_n$.
We will make use of the following observation regarding components of concordant links.
\begin{lemma}\label{L:components}\label{L:split-link}
Suppose that $L$ and $J$ are links with the same number of components and that $J$ has unknotted components. If $L$ is concordant to $\tilde J$, a
locally knotted $J$, then $L$ is concordant to~$J\# S(L)$.  In particular, if $L$ and $J$ have unknotted components and $L$ is concordant to a locally knotted $J$, then $L$ is concordant to~$J$.
\end{lemma}
\begin{proof}
Let $C$ be a concordance from $L$ to $\tilde J = J\# L'$ for some split link $L'$. Denote by $\widetilde C_j$ the concordance from the component $L_j'$ of $L'$ to the corresponding component $L_j$ of $L$ obtained by turning $C_j$ upside down. Take the product concordance from $\tilde J$ to itself, and sum each component $L_j' \times I$ with a copy of  $\widetilde C_j$, to obtain a concordance from $\tilde J$ to $J\# S(L)$.  Composing $C$ with this concordance gives a concordance from $L$ to~$J\# S(L)$.
\end{proof}

\begin{figure}[h]
\psfrag{L1}{$L_1$}
\psfrag{L2}{$L_2$}
\psfrag{L}[c][c]{$L$\kern5mm}
\psfrag{Jt}[c]{\kern10mm$\widetilde{J} = J \# L'$}
\psfrag{L1p}{$L_1'$}
\psfrag{L2p}{$L_2'$}
\psfrag{C1}{$C_1$}
\psfrag{C2}{$C_2$}
\psfrag{C1t}{$\widetilde{C_1}$}
\psfrag{C2t}{$\widetilde{C_2}$}
\psfrag{J}{$J\#S(L)$}
\includegraphics[scale=.7]{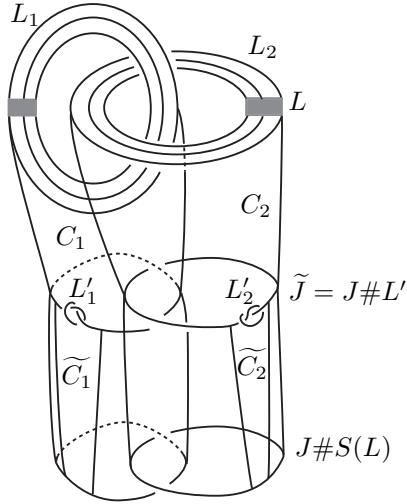}
\caption{Schematic illustration of proof of Lemma~\ref{L:components}}
\label{F:knotted-link}
\end{figure}

\section{Examples and proof of Theorem~\ref{T:alex1}}
\label{section:proof-main}

Consider the link $L = L(K)$ pictured in Figure~\ref{F:L(K)}, where
for the moment $K$ is an arbitrary knot in $S^3$ (cf.\ \cite[Figure~1]{Cha-Ko:1999-2}). The notation means
that the band labelled $K$ should be tied in the knot $K$, in such a
way that the framing of the band is~$0$.   The dotted curve $\alpha$
is not a component of the link, but is used in the description of
$L(K)$ as an `infection'.
\begin{figure}[h]\label{f:hopfexample}
\psfrag{L1}{$L_1$}
\psfrag{L2}{$L_2$}
\psfrag{al}{$\alpha$}
\psfrag{K}{$K$}
\includegraphics[scale=.9]{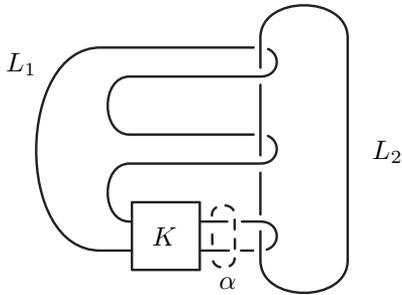}
\caption{The link $L(K)$}
\label{F:L(K)}
\end{figure}

\subsection*{Alexander polynomial of $L(K)$}

\begin{proposition} The link $L(K)$ has the following properties:
\begin{enumerate}
\item Both components $L_1$ and $L_2$ are unknotted.
\item The multivariable Alexander polynomial $\Delta_{L(K)}(x,y) = 1$.
\end{enumerate}
\end{proposition}
\begin{proof}
  (1) is obvious.  Observe that if $K$ were the unknot, then $L(K)$ is
  just the Hopf  link.  Also, $L(K)$ is obtained from the Hopf
  link by removing a tubular neighborhood of $\alpha$ and filling it
  in with the exterior of~$K$; in the resulting 3-manifold, which is
  $S^3$, the Hopf link becomes~$L(K)$.  It is well known that this operation
  preserves the homology of the universal cover of the link exterior
  (and consequently the Alexander polynomial) if $\alpha$ is
  null-homologous in the link complement.  For example, to show this one may apply
  Mayer-Vietoris and uses that a knot exterior is a homology $S^1\times
  D^2$. It follows that $L(K)$ has Alexander polynomial one.
\end{proof}

We remark that to prove (2) above, one may use the method of
D.~Cooper~\cite{cooper:abelian-cover}, which computes the Alexander
polynomial from matrices defined from Seifert surfaces
for the two components having only clasp-type intersections.  In our case, one can use the surfaces $S_1$ and $S_2$ in Figure~\ref{F:alex}.
\begin{figure}[h]
\psfrag{L1}[c]{$L_1$}
\psfrag{L2}[c]{$L_2$}
\psfrag{B1}{$S_1$}
\psfrag{B2}{$S_2$}
\psfrag{K}[c][c]{$K$\kern.5mm}
\psfrag{a}[r]{$a$}
\psfrag{b}{$b$}
\includegraphics[scale=.6]{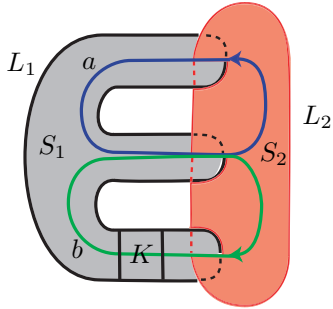}
\caption{Surfaces for Cooper's method}
\label{F:alex}
\end{figure}
The recipe in~\cite{cooper:abelian-cover} is to take a basis for $H_1(S_1 \cup S_2)$, and then derive a Seifert-type matrix recording linking numbers
amongst suitable pushoffs of the curves in this basis.  In our case, the curves $a$ and $b$ in Figure~\ref{F:alex} form such a basis. Since the linking numbers of the relevant pushoffs of $a$ and $b$ are independent of the choice of $K$, it follows that the Alexander polynomial of $L(K)$ is equal to that of $L(\text{unknot})$ which is the Hopf link.  Namely, $\Delta_{L(K)}=1$ for any~$K$.

\subsection*{Proof of Theorem~\ref{T:alex1}}

We present two proofs of Theorem~\ref{T:alex1} based on the
Ozsv{\'a}th-Szab{\'o} $\tau$ invariant~\cite{Ozsvath-Szabo:2003-1},
using the topological mechanisms described in
Section~\ref{section:covering-links-and-blowdown}.  In the following
section, we will give a third proof, using the $d$-invariant (or
correction term~\cite{oz:boundary}) that yields a stronger result
(stated as Theorem~\ref{T:infinite}).

\begin{proof}[First proof of Theorem~\ref{T:alex1}]
  Let $L=L(K)$ be the link
  illustrated in Figure~\ref{F:L(K)} where $K$ is a knot with positive
  $\tau$-invariant.  For instance, $K$ could be chosen to be the
  right-handed trefoil.  Denote the components of $L$ by $L_2$ and
  $L_1$ as in Figure~\ref{F:L(K)}.  We consider the covering knot $J$
  of $L$ obtained by taking the double cover of $S^3$ branched along
  the component~$L_2$; $J$~is the pre-image of $L_1$ in the
  resulting~$S^3$.  A standard cut-paste argument along the obvious
  2-disk bounded by $L_2$ shows that $J$ has the knot type of the
  positive Whitehead double $\Wh(K\#K^r)$, where $K^r$ denotes the
  orientation reverse of~$K$.

  From the hypothesis we have $\tau(K\#K^r)=2\tau(K)>0$.  Therefore by
  a result of Manolescu-Owens \cite{manolescu-owens:delta}, $J=\Wh(K\#
  K^r)$ is not rationally slice.  (See also
  Hedden~\cite{Hedden:2007-1}, which gives $\tau(J)=1$.)  Consequently
  $L$ is not concordant to the Hopf link, by
  Corollary~\ref{corollary:covering-knot-hopf-link}.

  Since the components of $L$ are unknotted, and $L$ is not concordant
  to the Hopf link, Lemma~\ref{L:components} implies that it is in
  fact not concordant to a locally knotted Hopf link.
\end{proof}

\begin{proof}[Second proof of Theorem~\ref{T:alex1}]
We will show that $L=L(K)$ is not concordant to the Hopf link,
making use of a recent $\tau$-invariant calculation by Adam
Levine~\cite{levine:doubles}. He considers a generalized Whitehead
double $D_{J,s}(K,t)$, defined (roughly) as a plumbing of two
annuli, tied into knots $J$ and $K$, with $s$ and $t$ twists
respectively. The case $s=-1$ and $J = \OO$ (the unknot) corresponds
to the $t$-twisted positive Whitehead double of~$K$. The knot in
Figure~\ref{F:blowdown} is the knot $D_{\OO,-2}(K,0)$, and
Levine~\cite[Proposition 2.5]{levine:doubles} computes its
$\tau$-invariant to be
\begin{equation}\label{E:tau}
D_{\OO,-2}(K,0) = \begin{cases}
0 & \text{if\ } \tau(K) \leq 0,\\
1 & \text{if\ } \tau(K) > 0.
\end{cases}
\end{equation}

Again we choose
$K$ to be a knot with $\tau(K)> 0$ for~$L$. Let $\lc_1$ be the knot
in $S^3$ obtained from $L$ by blowing down $L_2$ with positive
framing. One can easily see that $\lc_1$ is the knot
$D_{\OO,-2}(K,0)$. If $L$ were concordant to the Hopf link, then as
in Corollary~\ref{C:blowdown}, the knot $\lc_1$ would be
homotopically slice. However, by Levine's calculation, $\tau(\lc_1)
= 1$, which means that $\lc_1$ is not homotopically slice.
\begin{figure}[h]
\psfrag{K}[c][c]{\kern1.3mm$K$}
\includegraphics[scale=1]{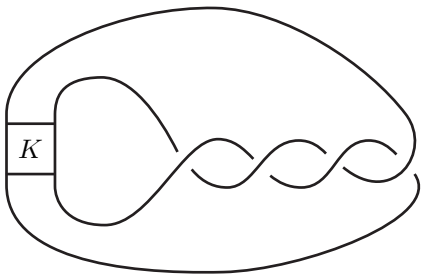}
\caption{$L(K)$ after blowing down $L_2$}
\label{F:blowdown}
\end{figure}
\end{proof}

A minor variation on the second proof may be obtained using the work of Rudolph~\cite{rudolph:quasi-slice}, coupled with the observation of Livingston~\cite{livingston:tau} that the results of~\cite{rudolph:quasi-slice} apply as well to surfaces lying in a homology ball, rather than~$B^4$.   The knot $\lc_1$ pictured in Figure~\ref{F:blowdown} may be described as the boundary of the plumbing of two annuli. One consists of two parallel copies of the knot $K$, with linking number $0$, and the other is an unknotted annulus that twists $-2$ times about its core.  As in the proof of~\cite[Lemma 2]{rudolph:quasi-slice}, if $K$ is strongly quasi-positive, then the knot $\lc_1$ is also strongly quasi-positive.   So if we choose $K$ to be non-trivial and strongly quasi-positive, then $\lc_1$ is not homotopically slice, and hence $L(K)$ is not concordant to the Hopf link.

\section{Infinitely many concordance classes}

By varying the choice of $K$, we can obtain infinitely many examples of links that have the properties stated in Theorem~\ref{T:alex1}. The verification that these links are not concordant to one another gives an alternate proof of Theorem~\ref{T:alex1}.
\begin{theorem}\label{T:infinite}
There is a sequence of knots $K(n)\  (n \in \N)$ such that the links
$L(K(n))$ satisfy the conclusion of Theorem~\ref{T:alex1} and are
distinct up to smooth concordance.
\end{theorem}
\begin{proof}
For any knot $K$, consider the knot $\lc_1(K)$ obtained by blowing
down the second component of~$L(K)$; there is an evident genus $1$
Seifert surface for $\lc_1(K)$ visible in Figure~\ref{F:blowdown}.
Using the algorithm of Akbulut and
Kirby~\cite{akbulut-kirby:branch}, it is then easy to draw a surgery
picture for  $M_2(\lc_1(K))$, the double cover of $S^3$ branched over $\lc_1(K)$, illustrated below.
\begin{figure}[h]
\psfrag{K}[c][c]{\kern1.2mm$K$}
\psfrag{Kr}[c][c]{\kern2.5mm$K^r$}
\psfrag{0}{$0$}
\psfrag{m4}{$-4$}
\includegraphics[scale=.8]{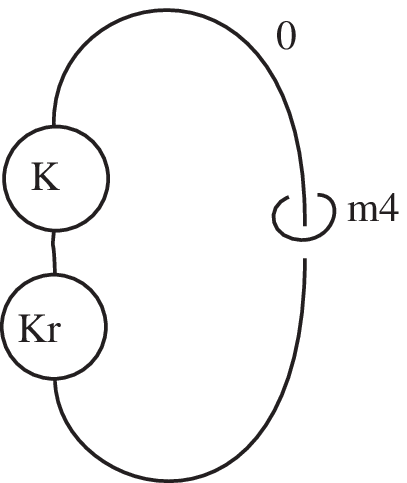}
\label{F:hopfbranch}
\caption{Double cover of $(\lc_1(K))$}
\end{figure}
By the slam-dunk move~\cite{gompf-stipsicz:book} this is diffeomorphic to $S^3_{1/4}(K\# K^r)$, a homology $3$-sphere.

Let $K(n)$ be the
$(2,2n+1)$ torus knot, and write $L(n)$ for~$L(K(n))$. Suppose that
$L(m)$ and $L(n)$ are concordant. Blow down the concordance between
$L_2(m)$ and $L_2(n)$ to get a concordance in a homotopy $S^3 \times
I$ between $\lc_1(m)$ and~$\lc_1(n)$.  Now, we take the $2$-fold
branched cover over that concordance, to get a $\Z_2$-homology
cobordism between $M_2(\lc_1(m))$ and~$M_2(\lc_1(n))$.  This implies
that the $d$-invariants (or correction term~\cite{oz:boundary}) of
these manifolds, in the trivial $\spinc$ structure, are equal.

The torus knots are reversible, and alternating so that $d$-invariants of their surgeries may be computed from their Alexander polynomials and signatures.  In particular, from \cite[Corollary 9.14]{oz:boundary} and \cite[Corollary 1.5]{oz:alternating} we get
$$d(S^3_{1/4}(K(n) \# K(n)^r)) = d(S^3_1(K(n) \# K(n)^r)) = -2n.$$
It follows that for positive $m \neq n$, the links $L(m)$ and $L(n)$
are not concordant and also not concordant to the Hopf link. Again
by Lemma~\ref{L:components}, $L(n)$ with $n>0$ is not concordant to
any locally knotted Hopf link.
\end{proof}

\def\cprime{$'$}
\providecommand{\bysame}{\leavevmode\hbox to3em{\hrulefill}\thinspace}
\def\MR#1{}
\providecommand{\MRhref}[2]{%
  \href{http://www.ams.org/mathscinet-getitem?mr=#1}{#2}
}
\providecommand{\href}[2]{#2}

\end{document}